%% file: trees2.tex
\documentclass{article}

\input{header}

\usepackage{pstricks}

\title{Uniform geometric estimates for sublevel sets}
\author{Philip T. Gressman}
\begin{document}
\maketitle

\begin{abstract}
This paper reconsiders the uniform sublevel set estimates of Carbery, Christ, and Wright \cite{ccw1999} and Phong, Stein, and Sturm \cite{pss2001} from a geometric perspective.  This perspective leads one to consider a natural collection of homogeneous, nonlinear differential operators which generalize mixed derivatives in $\R^d$.  As a consequence, it is shown that, in the case of both of these previous works, improved uniform decay rates are possible in many situations.
\end{abstract}

\section{Introduction}

Many of the most basic questions in harmonic analysis are concerned with understanding one or more geometric invariants of a function $F(x)$ defined on some neighborhood in $\R^d$ (two of the most common being the growth rate of the measure of its sublevel sets or the decay rate of the scalar oscillatory integral having this function as the phase).  There has been considerable interest in understanding the extent to which the Newton polyhedron (and associated Newton distance) of $F$ determines or fails to determine the sharp values of these invariants.  Var\v{c}enko  \cite{varcenko1976}, Phong and Stein \cite{ps1997}, Phong, Stein, and Sturm \cite{pss2001}, Ikromov, Kempe, and M\"{u}ller \cite{ikm2007}, and Magyar \cite{magyar2009} are just a few of the many examples which have arisen over the last thirty years.  Closely related to these questions are issues of stability and the extent to which these invariants are left unchanged after some perturbation.  In many cases it is necessary to impose some sort of ``smallness'' condition on the perturbation; see Karpushkin \cite{karpushkin1986}, Phong, Stein, Sturm \cite{pss1999}, Phong, Sturm \cite{ps2000}, Greenblatt \cite{greenblatt2005mrl}, and Seeger \cite{seeger1998} for examples of this sort.  

The last decade has also witnessed the emergence of a different sort of stability estimate, in which uniform bounds for the relevant invariants are deduced given only lower bounds on certain derivatives of $F$ (so in particular, there are no ``smallness'' assumptions).  The two most influential works in this direction are due to to Carbery, Christ, and Wright \cite{ccw1999} and Phong, Stein, and Sturm \cite{pss2001}.  In the former paper, stability estimates for sublevel sets and oscillatory integrals for smooth $F$ are established; it turns out that these estimates depend only on the minimum value of $|\beta|$ as $\beta$ ranges over all multiindices $\beta$ with $|\partial^\beta F| \geq 1$ on the unit box, meaning that there is no improvement of the index when many derivatives are simultaneously nonzero in addition to the particular derivative of minimal order.  Examples of the form $F = p \circ \ell$ (where $p$ is a polynomial in one variable and $\ell$ is linear) demonstrate that, in at least some cases, the dependence of the invariants on only one multiindex is genuine and not an artifact of the method of proof (since in this case, the sharp estimates for the one-dimensional function $p$ continue to be sharp for $p \circ \ell$, but many mixed derivatives will be simultaneously nonzero for generic $\ell$).
The paper of Phong, Stein, and Sturm \cite{pss2001}, on the other hand,  establishes uniform estimates for the norm decay rates of sublevel-set functionals and does not encounter this sort of limitation; in particular, their main results may be formulated in terms of a Newton distance (or restricted Newton distance) associated to $F$ (and so, in particular, there are generally improved estimates when many derivatives are simultaneously nonzero).  The only downside is that these estimates are only guaranteed to be sharp on products $L^{p_1}(\R) \times \cdots \times L^{p_d}(\R)$ in the range $d-1 \leq \sum_{j=1}^d \frac{1}{p_j} \leq d$.

The purpose of the present paper is to demonstrate that, in both the scalar case of Carbery, Christ, and Wright \cite{ccw1999} and the multilinear case of Phong, Stein, and Sturm \cite{pss2001}, there are additional geometric mechanisms which allow for improved estimates in the presence of multiple nonvanishing derivatives.  The key is to require nonvanishing of certain nonlinear expressions involving derivatives of $F$ rather than to require nonvanishing of those same derivatives in independent ways.  These nonlinear expressions arise naturally when considering sublevel and oscillatory integral problems from a geometrically invariant viewpoint.  In the scalar case, they allow for sharper estimates whenever $F$ is not of the exceptional form already mentioned (that is, some one-dimensional function composed with a linear mapping).  In the multilinear case, they provide improved results outside the range $d-1 \leq \sum_{j=1}^d \frac{1}{p_j} \leq d$.  In neither, case, though, do the estimates correspond in an elementary way to a Newton distance.

The class of differential operators relevant to Euclidean sublevel set and oscillatory integral problems may be defined in the following way:
\begin{definition}
A differential operator on functions in $\R^d$ will be called admissible of type $(\alpha,\beta)$ for some positive integer $\alpha$ and some multiindex $\beta$ of length $d$ provided that it lies in the class ${\cal L}^{\alpha,\beta}$ described inductively by:
\begin{enumerate}
\item Let ${\cal L}^{1,(0,\ldots,0)}$ be the set containing only the identity operator on smooth functions in $\R^d$ (i.e., $L F = F$ for all smooth $F$ and $L \in {\cal L}^{1,(0,\ldots,0)}$).
\item For $(\alpha,\beta) \neq (1,(0,\ldots,0))$, ${\cal L}^{\alpha,\beta}$ consists of all operators $L$ given by
\begin{equation} L F := \det \left[ \begin{array}{ccc}
\frac{\partial}{\partial x_{i_1}} L_1 F & \cdots & \frac{\partial}{\partial x_{i_1}} L_n F \\
 \vdots & \ddots & \vdots \\
\frac{\partial}{\partial x_{i_n}} L_1 F & \cdots & \frac{\partial}{\partial x_{i_n}} L_n F
\end{array} \right] \label{opind}
\end{equation}
for some $L_1 \in {\cal L}^{\alpha_1,\beta_1},\ldots, L_n \in {\cal L}^{\alpha_n,\beta_n}$ satisfying $\alpha = \alpha_1 + \cdots + \alpha_n$ and $\beta = \beta_1 + \cdots + \beta_n + I$, where $I$ is the multiindex $(I_1,\ldots,I_d)$ for which $I_{i_1} = \cdots = I_{i_n} = 1$ and $I_j = 0$ for all remaining indices $j$.
\end{enumerate}
\end{definition}
Observe that the admissible operators of type $(1,\beta)$ are exactly the mixed partial derivative operators $\pm \partial^{\beta}$. For $\alpha > 1$, ${\cal L}^{\alpha,\beta}$ will consist of many distinct operators (none of which are linear); for example, both of the following operators belong to ${\cal L}^{2,(2,2)}$:
\[ F \mapsto \frac{\partial^2 F}{\partial x^2} \frac{\partial^2 F}{\partial y^2} - \left(\frac{\partial^2 F}{\partial x \partial y} \right)^2 \ \mbox{  and  }  F \mapsto \ \frac{\partial F}{\partial x} \frac{\partial^3 F}{\partial x \partial y^2} - \frac{\partial F}{\partial y} \frac{\partial^3 F}{\partial x^2 \partial y}. \]
The Euclidean versions of the main results of this paper may now be stated in terms of the classes ${\cal L}^{\alpha,\beta}$:  
\begin{theorem}[cf. Theorem 1.3 of \cite{ccw1999}]
Suppose $F$ is a real-analytic function on some connected, open $U \subset \R^d$ containing $[0,1]^d$.  Fix any closed set $D \subset [0,1]^d$.  If $L$ is an admissible differential operator of type $(\alpha,\beta)$ and $\epsilon > 0$, then
\[  \left| \set{ x \in D}{ |F(x)| \leq \epsilon} \right| \leq C \epsilon^{\frac{\alpha}{|\beta|+1-\alpha}} \left( \inf_{x \in D} |L F(x)| \right)^{-\frac{1}{|\beta|+1-\alpha}}  \]
for some constant $C$ independent of $\epsilon$.  If $F$ is a Pfaffian function on $U$, then the constant $C$ depends only on $d$, $L$, and the format of $F$. \label{ccwthm1}
\end{theorem}
\begin{theorem}[cf. Theorem 1.4 of \cite{ccw1999}]
Suppose $F$ is a real-analytic function on some connected, open $U \subset \R^d$ containing $[0,1]^d$.  Let $D \subset [0,1]^d$ be any semi-analytic set.  If $L$ is an admissible operator of type $(\alpha,\beta)$ and $\lambda \in \R$, then
\[ \left| \int_D e^{i \lambda F(x)} dx \right| \leq C |\lambda|^{-\frac{\alpha}{|\beta| + 1 - \alpha}} \left(   \inf_{x \in D} |L F(x)| \right)^{-\frac{1}{|\beta|+1-\alpha}} \]
for some constant $C$ independent of $\epsilon$.  If $F$ is Pfaffian on $U$ and $D$ is semi-Pfaffian, then $C$ depends only on $d$, $L$, and the formats of $F$ and $D$. \label{ccwthm2}
\end{theorem}

\begin{theorem}[cf. Theorem A of \cite{pss2001}]
Suppose $F$ is a real-analytic function on some connected, open $U \subset \R^d$.  Fix any compact set $D \subset U$.  If $L$ is any admissible operator of type $(\alpha,\beta)$ and $\epsilon > 0$, then
\[
 \left| \int_D \chi_{|F(x)| \leq \epsilon} \prod_{i=1}^d f_i(x_i) dx \right| \leq C \epsilon^{\frac{\alpha}{|\beta|+1-\alpha}} \left( \inf_{x \in D} |L F(x)| \right)^{-\frac{1}{|\beta|+1-\alpha}} \prod_{i=1}^d ||f_i||_{L^{p_i,1}} \]
where $\frac{1}{p_i} = 1- \frac{\beta_i}{|\beta| + 1 - \alpha}$ ($L^{p,1}(\R)$ is the strong-type Lorentz space). The constant $C$ does not depend on $\epsilon$, and if $F$ is Pfaffian on $U$, then $C$ depends only on $d$, $L$, and the format of $F$ (in particular, $C$ does not depend on $D$ in this case). \label{pssthm2}
\end{theorem}
A remark about the appearance of Pfaffian functions:  unlike the situation of Carbery, Christ, and Wright, elementary counterexamples show that the strongest uniformity results fail in the $C^\infty$ category for essentially topological reasons.  It appears that the most general class of functions for which strong uniformity results hold in this context are the Pfaffian functions.  For the moment, Pfaffian functions may be regarded as a generalizations of polynomials, and the format be regarded as a generalization of the degree.

Theorems \ref{ccwthm1}, \ref{ccwthm2}, and \ref{pssthm2} will be obtained as corollaries of theorems \ref{maintheorem} and \ref{oscthm}, which are geometric formulations of the sublevel set and oscillatory integral problems.  Roughly speaking, the idea is to treat the coordinate functions $x_1,\ldots,x_d$ and the function $F(x)$ as equals.  The derivative $\frac{\partial}{\partial x_i}$ is thought of as the unique vector field on $\R^d$ which annihilates $x_j$ for $j \neq i$ (properly normalized). Likewise, vector fields which annihilate some of the coordinate functions as well as the function $F$ will arise; these vector fields played no role in earlier work.  These vector fields lead to the classes ${\cal L}^{\alpha,\beta}$ and geometric generalizations.  

To that end, suppose that ${\cal M}$ is a $d$-dimensional real-analytic, compact, oriented manifold with boundary, and let $\omega$ be a real-analytic, nonnegative $d$-form on ${\cal M}$.  Now let $\pi_1,\ldots,\pi_m$ be real-analytic functions on ${\cal M}$, let $D \subset {\cal M}$ be closed, and consider the multilinear functional $T_{\pi,D}$ given by
\begin{equation}
T_{\pi,D}(f_1,\ldots,f_m) := \int_{D} \left( \prod_{i=1}^m f_i \circ \pi_i \right) \omega, \label{multilin}
\end{equation}
where $f_1,\ldots,f_m$ are locally integrable functions on $\R$.  The operators generalizing the classes ${\cal L}^{\alpha,\beta}$ will be called $d$-trees on $m$-indices (they correspond in a natural way to certain labeled, directed graphs) and are defined as follows:
\begin{definition}
Let ${\cal G}_0$ be the set $\{1,\ldots,m\}$.  For all integers $k \geq 1$, let ${\cal G}_k := {\cal G}_{k-1} \cup ({\cal G}_{k-1})^d$ (that is, the union of ${\cal G}_{k-1}$ and elements of the $d$-fold Cartesian product of ${\cal G}_{k-1}$'s).  Any element $G \in \bigcup_{k=0}^\infty {\cal G}_k$ will be called a $d$-tree on $m$ indices; all $G \in {\cal G}_0$ will be called trivial.
For any $G \in \bigcup_{k=0}^\infty {\cal G}_k$, define constants $\#G$, $G^{(1)},\ldots,G^{(m)}$, and the operator $\partial^G$ as follows:
\begin{enumerate}
\item If $G$ is trivial, define $\#G := 0$, and for each $i=1,\ldots,m$, define $G^{(i)} := 1$ if $G = i$ and $G^{(i)} := 0$ otherwise. Given functions $\pi_1,\ldots,\pi_m$,  define $\partial^G \pi := \pi_G$ (that is, $\partial^G \pi$ is the $G$-th function).
\item If $G \in {\cal G}_k \setminus {\cal G}_{k-1}$ for some $k \geq 1$, then $G = (G_1,\ldots,G_d)$ for some $G_1,\ldots,G_d \in {\cal G}_{k-1}$.  In this case, define $\# G := \# G_1 + \cdots + \# G_d + 1$ and $G^{(i)} := G^{(i)}_1 + \cdots + G^{(i)}_d$.  Given $\pi_1,\ldots,\pi_m$, define $\partial^G \pi$ to be the unique function on ${\cal M}$ which satisfies 
\[ d (\partial^{G_1} \pi) \wedge \cdots \wedge d (\partial^{G_d} \pi) = \left(\partial^{G} \pi \right) \omega. \]
\end{enumerate}
\end{definition}
With these definitions, the main theorems of this paper are stated as follows:
\begin{theorem}
Let ${\cal M}, \omega$, and $\pi_1,\ldots,\pi_m$ be as already defined. Let $D \subset {\cal M}$ be closed.  
For any $K$ and $G \in {\cal G}_K \setminus {\cal G}_0$, there exists a finite constant $C$ such that \label{maintheorem}
\[ \left| T_{\pi,D}(f_1,\ldots,f_m) \right| \leq C \left( \inf_{x \in D} |\partial^G \pi(x)| \right)^{-\frac{1}{\#G}} \prod_{j=1}^m ||f_j||_{L^{p_j,1}} \]
for any locally integrable functions $f_1,\ldots,f_m$, where $\frac{1}{p_j} = \frac{G^{(j)}}{\#G}$ for $j=1,\ldots,m$. 
If ${\cal M} \subset \R^d$, $\omega = dx_1 \wedge \cdots \wedge dx_d$, and the functions $\pi_j$ are Pfaffian on some connected, open set $U$ containing ${\cal M}$,  then 
the constant $C$ depends only on $d$, $K$, $m$, and the formats of $\pi_1,\ldots,\pi_m$.
\end{theorem}

\begin{theorem}
Let ${\cal M}, \omega$, and $\pi_1,\ldots,\pi_m$ be as defined above. Fix any $K$ and any $G \in {\cal G}_K \setminus {\cal G}_0$.  Let $D$ be any semi-analytic set in ${\cal M}$.  There exists a finite constant $C$ such that \label{oscthm}
\[ \left| \int_{E_\epsilon \cap D} e^{i \lambda \pi_n} \ \omega \right| \leq C \left( \inf_{x \in E_\epsilon \cap D} |\partial^G \pi(x)| \right)^{- \frac{1}{\# G}} |\lambda|^{-\frac{G^{(m)}}{\#G}} \prod_{j=1}^{m-1} \epsilon_j^{\frac{G^{(j)}}{\#G}}, \]
for any real $\lambda$ and any positive $\epsilon_1,\ldots,\epsilon_{m-1}$, where
\[ E_\epsilon := \set{ x \in {\cal M}}{ |\pi_j(x)| \leq \epsilon_j \ j=1,\ldots,m-1}. \]
If ${\cal M} \subset \R^d$, $\omega = dx_1 \wedge \cdots \wedge dx_d$, and the functions $\pi_j$ are Pfaffian on some connected, open set $U$ containing ${\cal M}$, and $D$ is semi-Pfaffian,  then the constant $C$ depends only on $d$, $K$, $m$, the formats of  $\pi_1,\ldots,\pi_m$, and the format of $D$.
\end{theorem}

The proofs of theorems \ref{maintheorem} and \ref{oscthm} are contained in sections \ref{proofsec} and \ref{oscsec}, respectively.  The methods used are extremely elementary and circumvent the need for the intricate decompositions performed by Phong, Stein, and Sturm (while retaining the induction argument and repeatedly localizing to the set where all derivatives are small until this set must necessarily be empty).  Section \ref{reducesec} illustrates the connection between the operators ${\cal L}^{\alpha,\beta}$ and a corresponding family of $d$-trees on $d+1$ indices, and derives theorems \ref{ccwthm1}, \ref{ccwthm2}, and \ref{pssthm2} from theorems \ref{maintheorem} and \ref{oscthm}.

\section{Reduction to theorems \ref{maintheorem} and \ref{oscthm}}
\label{reducesec}

It may be helpful to consider a natural directed graph which can be associated to each $d$-tree, representing its overall ``shape.''  If If $G$ is trivial, its directed graph consists of one vertex (called the root) and no edges.  If $G = (G_1,\ldots,G_d)$ is nontrivial, then the graph associated to $G$ is obtained by taking a disjoint union of the graphs corresponding to $G_1,\ldots, G_d$, adding another vertex to this union (the new root), and adding $d$ additional edges, each directed from the new root to the old roots of $G_1,\ldots,G_d$. 
Note that $G^{(1)} + \cdots + G^{(m)}$ equals the number of vertices in the shape graph of $G$ which have no children (called the leaves of the graph) and $\#G + G^{(1)} + \cdots + G^{(m)}$ equals the total number of vertices in the shape graph. 

\begin{figure}
\centering
\psset{unit=4.9cm}
\begin{pspicture}(0,-0.05)(2.125,0.9)
\pscircle*(0.45,0.9){0.015}

\psline[arrows=->,arrowsize=0.045,arrowinset=0](0.45,0.9)(0.0,0.6)
\psline[arrows=->,arrowsize=0.045,arrowinset=0](0.45,0.9)(0.15,0.6)
\psline[arrows=->,arrowsize=0.045,arrowinset=0](0.45,0.9)(0.3,0.6)
\psline[arrows=->,arrowsize=0.045,arrowinset=0](0.45,0.9)(0.45,0.6)
\psline[arrows=->,arrowsize=0.045,arrowinset=0](0.3,0.6)(0.0,0.3)
\psline[arrows=->,arrowsize=0.045,arrowinset=0](0.3,0.6)(0.15,0.3)
\psline[arrows=->,arrowsize=0.045,arrowinset=0](0.3,0.6)(0.3,0.3)
\psline[arrows=->,arrowsize=0.045,arrowinset=0](0.3,0.6)(0.45,0.3)
\psline[arrows=->,arrowsize=0.045,arrowinset=0](0.15,0.3)(0.0,0)
\psline[arrows=->,arrowsize=0.045,arrowinset=0](0.15,0.3)(0.15,0)
\psline[arrows=->,arrowsize=0.045,arrowinset=0](0.15,0.3)(0.3,0)
\psline[arrows=->,arrowsize=0.045,arrowinset=0](0.15,0.3)(0.45,0)

\pscircle*(0.0,0.6){0.015} 
\pscircle*(0.15,0.6){0.015}
\pscircle*(0.3,0.6){0.015} 
\pscircle*(0.45,0.6){0.015}
\pscircle*(0.0,0.3){0.015} 
\pscircle*(0.15,0.3){0.015}
\pscircle*(0.3,0.3){0.015} 
\pscircle*(0.45,0.3){0.015}
\pscircle*(0.0,0){0.015} 
\pscircle*(0.15,0){0.015}
\pscircle*(0.3,0){0.015} 
\pscircle*(0.45,0){0.015}

\uput[-90](0.0,0.6){1} 
\uput[-90](0.15,0.6){2}
\uput[-90](0.45,0.6){4}
\uput[-90](0.0,0.3){1} 
\uput[-90](0.3,0.3){3} 
\uput[-90](0.45,0.3){4}
\uput[-90](0.0,0){5}
\uput[-90](0.15,0){2}
\uput[-90](0.3,0){3} 
\uput[-90](0.45,0){4}

\pscircle*(1.5,0.9){0.015}
\psline[arrows=->,arrowsize=0.045,arrowinset=0](1.5,0.9)(1,0.45)
\psline[arrows=->,arrowsize=0.045,arrowinset=0](1.5,0.9)(1.333,0.45)
\psline[arrows=->,arrowsize=0.045,arrowinset=0](1.5,0.9)(1.667,0.45)
\psline[arrows=->,arrowsize=0.045,arrowinset=0](1.5,0.9)(1.999,0.45)
\psline[arrows=->,arrowsize=0.045,arrowinset=0](1,0.45)(0.875,0.0)
\psline[arrows=->,arrowsize=0.045,arrowinset=0](1,0.45)(0.9583,0.0)
\psline[arrows=->,arrowsize=0.045,arrowinset=0](1,0.45)(1.0417,0.0)
\psline[arrows=->,arrowsize=0.045,arrowinset=0](1,0.45)(1.125,0.0)
\psline[arrows=->,arrowsize=0.045,arrowinset=0](1.333,0.45)(1.2083,0.0)
\psline[arrows=->,arrowsize=0.045,arrowinset=0](1.333,0.45)(1.2917,0.0)
\psline[arrows=->,arrowsize=0.045,arrowinset=0](1.333,0.45)(1.375,0.0)
\psline[arrows=->,arrowsize=0.045,arrowinset=0](1.333,0.45)(1.4583,0.0)
\psline[arrows=->,arrowsize=0.045,arrowinset=0](1.667,0.45)(1.5417,0.0)
\psline[arrows=->,arrowsize=0.045,arrowinset=0](1.667,0.45)(1.625,0.0)
\psline[arrows=->,arrowsize=0.045,arrowinset=0](1.667,0.45)(1.7083,0.0)
\psline[arrows=->,arrowsize=0.045,arrowinset=0](1.667,0.45)(1.7917,0.0)
\psline[arrows=->,arrowsize=0.045,arrowinset=0](1.999,0.45)(1.875,0.0)
\psline[arrows=->,arrowsize=0.045,arrowinset=0](1.999,0.45)(1.9583,0.0)
\psline[arrows=->,arrowsize=0.045,arrowinset=0](1.999,0.45)(2.0417,0.0)
\psline[arrows=->,arrowsize=0.045,arrowinset=0](1.999,0.45)(2.125,0.0)

\uput[-90](0.875,0.0){5}
\uput[-90](0.9583,0.0){2}
\uput[-90](1.0417,0.0){3}
\uput[-90](1.125,0.0){4}
\uput[-90](1.2083,0.0){1}
\uput[-90](1.2917,0.0){5}
\uput[-90](1.375,0.0){3}
\uput[-90](1.4583,0.0){4}
\uput[-90](1.5417,0.0){1}
\uput[-90](1.625,0.0){2}
\uput[-90](1.7083,0.0){5}
\uput[-90](1.7917,0.0){4}
\uput[-90](1.875,0.0){1}
\uput[-90](1.9583,0.0){2}
\uput[-90](2.0417,0.0){3}
\uput[-90](2.125,0.0){5}

\pscircle*(1,0.45){0.015}  \pscircle*(1.333,0.45){0.015}  \pscircle*(1.667,0.45){0.015}  \pscircle*(1.999,0.45){0.015}
\pscircle*(0.875,0.0){0.015}  \pscircle*(0.9583,0.0){0.015}  \pscircle*(1.0417,0){0.015}  \pscircle*(1.125,0){0.015}
\pscircle*(1.2083,0){0.015}  \pscircle*(1.2917,0){0.015}  \pscircle*(1.375,0){0.015}  \pscircle*(1.4583,0){0.015}
\pscircle*(1.5417,0){0.015}  \pscircle*(1.625,0){0.015}  \pscircle*(1.7083,0){0.015}  \pscircle*(1.7917,0){0.015}
\pscircle*(1.875,0){0.015}  \pscircle*(1.9583,0){0.015}  \pscircle*(2.0417,0){0.015}  \pscircle*(2.125,0){0.015}
\end{pspicture}
\caption{The shape graphs (with labels on the leaves) for the examples \eqref{mixedd} and \eqref{hessian}, respectively in the case $d=4$. 
}
\label{graphs}
\end{figure}
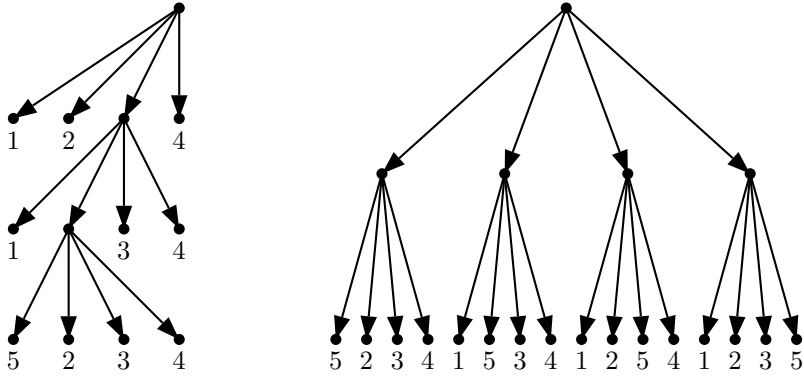

The theorems \ref{ccwthm1}, \ref{ccwthm2} and \ref{pssthm2} are derived from theorems \ref{maintheorem} and \ref{oscthm} in the following way.  First take ${\cal M}$ to be a closed Euclidean ball in $\R^d$ of some large radius and $\omega = dx_1 \wedge \cdots \wedge dx_d$.  Regarding the functions $\pi_j$, set $m = d+1$, take $\pi_j(x) = x_j$ for $j=1,\ldots,d$ and let $\pi_{d+1} = F$.   Consider the example of the $d$-tree $G$ given by 
\begin{equation}
 (1,2,(1,(d+1,2,\ldots,d),3,\ldots,d),4,\ldots,d) \label{mixedd}
\end{equation}
has $\partial^G \pi = \frac{\partial}{\partial x_3}  \frac{\partial}{\partial x_2}  \frac{\partial}{\partial x_1} F$.  Continuing this nesting process, any mixed derivative $\partial^\beta F$ for some multiindex $\beta$ may be realized as $\partial^G \pi$ for some $d$-tree which has $\#G = |\beta|$, $G^{(d+1)} = 1$ and $G^{(j)} = |\beta| - \beta_j$ for for all remaining indices.  
For the $d$-tree $G'$ given by
\begin{equation}
 ((d+1,2,\ldots,d),(1,i,3,\ldots,d),\ldots,(1,\ldots,d-1,d+1)), \label{hessian}
\end{equation}
on the other hand, $\partial^{G'} \pi$ is exactly the determinant of the Hessian $\partial^2_x F$.  In this case $\#G' = d+1$, $(G')^{(j)} = d-1$ for $j=1,\ldots,d$, $(G')^{(i)} = d$, and $(G')^{(j')} = 0$ otherwise.  
In general, given some $L \in {\cal L}^{\alpha,\beta}$, the function $L F$ will equal $\partial^G \pi$ for some $G$ which satisfies $\#G = |\beta|+1-\alpha$, has $G^{(i)} = \#G - \beta_i$ for $i=1,\ldots,d$ and $G^{(d+1)} = \alpha$.  This fact is readily established by induction upon noting that, if $L$ is related to $L_1,\ldots,L_n$ as described by \eqref{opind}, then
\[ (LF) \omega = d (L_1 f) \wedge \cdots \wedge d(L_n f) \wedge d x_{i_{n+1}} \wedge \cdots \wedge d x_{i_d} \]
where $i_{n+1},\ldots,i_d$ are suitably chosen indices taken from $\{1,\ldots,d\}$.  In particular, if each of $L_1,\ldots,L_n$ corresponds to a $d$-tree $G_1,\ldots,G_n$ with the asserted numerology, then $L$ corresponds to the $d$-tree $G := (G_1,\ldots,G_n,i_{n+1},\ldots,i_d)$ which satisfies
\begin{align*}
\#G & = \#G_1 + \cdots + \#G_n + 1 \\
 & = \left( |\beta_1| + \cdots + |\beta_n| + n \right) + 1  -  \left( \alpha_1 +\cdots + \alpha_n \right) = |\beta| + 1 - \alpha, \\
G^{(d+1)} & = G_1^{(d+1)} + \cdots + G^{(d+1)}_n \\
& = \left( \alpha_1 +\cdots + \alpha_n \right) = \alpha, \\
G^{(i)} & = G_1^{(i)} + \cdots + G^{(i)}_n  + \delta_{i_{n+1},i} + \cdots + \delta_{i_{d},i} \qquad (i \neq d+1) \\
& = |\beta| + 1 - \alpha - (\beta_1)_i - \cdots - (\beta_n)_i - \delta_{i_1,i} - \cdots - \delta_{i_n,i} \\
& = |\beta| + 1 - \alpha - \beta_i,
\end{align*}
where $\delta_{i,j}$ is the usual Kronecker delta.  With these numerical relationships in place, it only remains to observe that theorem \ref{pssthm2} follows \ref{maintheorem} for these special ${\cal M}$, $\omega$, and $\pi_j$'s when $f_{d+1} := \chi_{[-\epsilon,\epsilon]}$.  Theorem \ref{ccwthm1} will follow if, in addition, $f_1 = \cdots = f_{d} = \chi_{[0,1]}$.  Finally, note that theorem \ref{ccwthm2} follows from theorem \ref{oscthm} with these same substitutions.

It is perhaps also worth noting that another important special case of \eqref{multilin} occurs when ${\cal M}$ is a closed, Euclidean ball in $\R^d$ and $\omega = dx_1 \wedge \cdots \wedge dx_d$ as above,  but the functions $\pi_1,\ldots,\pi_m$ are linear.  In this case, the functional \eqref{multilin} is a H\"{o}lder-Brascamp-Lieb functional, the study of which was initiated by Brascamp and Lieb \cite{bl1976}. More recent work on such functionals is due to Bennett, Carbery, Christ, and Tao \cite{bcct2008}, \cite{bcct2005}, Bennett and Bez \cite{bb2009}, and others.  In this special sub-case of \eqref{multilin}, the entire problem reduces to a question of finding maximal subsets of $\{1,\ldots,m\}$ for which the corresponding functions $\pi_1,\ldots,\pi_m$ have linearly independent differentials $d \pi_1,\ldots, d \pi_m$.  When the functions $\pi_i$ are allowed to be nonlinear, complicated partial degeneracies arise (in which the differentials $d \pi_i$ may be linearly dependent at some points but not all) which were not encountered in any of these earlier situations.

\section{Proof of theorem \ref{maintheorem}}
\label{proofsec}

Given the pair $({\cal M},\omega)$, it will be necessary to address the issue of the topological complexity of the sublevel sets associated to these functions $\pi_1,\ldots,\pi_m$ as well as the derivatives $\partial^G \pi$.  This complexity appears indirectly in the proofs of theorems \ref{maintheorem} and \ref{oscthm} when one is forced to count nondegenerate solutions of various systems of equations.  To begin, consider the following definitions:
\begin{definition}
Any $d$-tree $G$ will be said to satisfy the weak multiplicity $N$ condition on ${\cal M}$ when either $G$ is trivial or $G = (G_1,\ldots,G_d)$ for some $G_1,\ldots,G_d$ such that, the system of equations $\partial^{G_1} \pi = c_1,\ldots,\partial^{G_d} \pi = c_d$ has at most $N$ nondegenerate solutions in ${\cal M}$ for any $(c_1,\ldots,c_d) \in \R^d$. 
\end{definition}

\begin{definition}
Any $d$-tree $G$ will satisfy the strong multiplicity $N$ condition on ${\cal M}$ if $G$ is trivial or if $G$ satisfies the weak multiplicity $N$ condition and $G = (G_1,\ldots,G_d)$ for some $G_1,\ldots,G_d$ which also satisfy the strong multiplicity $N$ condition on ${\cal M}$.  Let ${\cal G}_K^N({\cal M},\pi)$ be the collection of {\it nontrivial} $d$-trees $G \in {\cal G}_K$ which satisfy the strong multiplicity $N$ condition on ${\cal M}$.
\end{definition}
In other words, $G$ satisfies the strong multiplicity $N$ condition on ${\cal M}$ if it can be inductively built out of sub-trees all of which satisfy the weak multiplicity $N$ condition.
The class ${\cal G}_K^N({\cal M},\pi)$ will only {\it fail} to equal ${\cal G}_K \setminus {\cal G}_0$ when a system of equations arises involving the $\pi$'s and their derivatives which has more than $N$ nondegenerate solutions in ${\cal M}$.  In the real-analytic category, the compactness of ${\cal M}$ guarantees that this cannot happen for arbitrarily large $N$.  This is a consequence of a result by Gabri{e}lov \cite{gabrielov1968}, namely, that the number of connected components of an analytic family of semi-analytic sets is locally bounded.  In general, of course, the bound for $N$ will depend in some unknown way on the functions $\pi_j$.  In the special case of Pfaffian functions, though, effective bounds are possible, thanks to a theorem of Khovanski{\u\i} \cite{khovanskii1980} generalizing B\'{e}zout's theorem.  This class of functions is not commonly referenced in earlier works on sublevel sets or oscillatory integrals; it includes all elementary functions (on suitable domains), but not all analytic functions.  See Khovanski{\u\i} \cite{khovanskii1991} or Gabrielov and Vorobjov \cite{gv2004} for more in this area, or van den Dries \cite{vandendries1998} for extensions and generalizations in the context of o-minimality.

\begin{definition} 
A Pfaffian chain of the order $r \geq 0$ and degree $\alpha \geq 1$ in a connected, open set $U \subset \R^d$ is a sequence of analytic functions $f_1,\ldots,f_r$ in $G$ satisfying differential equations
\[ df_j(x) = \sum_{i=1}^d g_{ij}(x,f_1(x),\ldots,f_j(x)) dx_i \]
for $j=1,\ldots,r$, where $g_{ij} \in \R[x_1,\ldots,x_d,y_1,\ldots,y_j]$ has degree at most $\alpha$.  A Pfaffian function of order $r$ and degree $(\alpha,\beta)$ for $r \geq 0$ and $\alpha,\beta \geq 1$ is any function expressible as $f(x) = P(x,f_1(x),\ldots,f_r(x))$ for some polynomial $P$ of degree at most $\beta$ and some Pfaffian chain $f_1,\ldots f_r$ of order $r$ and degree $\alpha$.
\end{definition}

\begin{theorem}[Khovanski{\u\i} \cite{khovanskii1980}]
Suppose $f_1,\ldots,f_d$ are Pfaffian functions in a connected, open set $U \subset \R^d$ having a common Pfaffian chain of order $r$ and degrees $(\alpha,\beta_i)$.  Then the number of nondegenerate solutions of the system $f_1(x) = \cdots = f_d(x) = 0$ in $U$ does not exceed
\[ 2^{r(r-1)/2} \beta_1 \cdots \beta_d \left( \min\{d,r\} \alpha + \beta_1 + \cdots + \beta_d - d + 1 \right)^r. \]
\end{theorem}

The main analytic tool relating $d$-trees to the estimation of \eqref{multilin} comes in the form of the following proposition.  Perhaps unsurprisingly, it is nothing more than a geometric version of the H\"{o}lder/Fubini argument that one would use in the case of H\"{o}lder-Brascamp-Lieb inequalities:
\begin{proposition}
Suppose ${\cal M}$ is an oriented $d$-dimensional manifold.  \label{covprop} Suppose $\pi_1,\ldots,\pi_d$ are $C^1$ functions on ${\cal M}$.  For any $t = (t_1,\ldots,t_d) \in \R^d$, let $N_\pi(t)$ be the number of nondegenerate solutions $x \in {\cal M}$ of the system of equations $\pi_1(x) = t_1,\ldots, \pi_d(x) = t_d$. Suppose $\psi \in L^\infty({\cal M})$ and $\phi$ is a measurable function on $\R^d$.  Then \label{cov}
\[ \left| \int_{\cal M} \phi(\pi_1,\ldots,\pi_d) \psi d \pi_1 \wedge \cdots \wedge d \pi_d \right| \leq  ||\psi||_\infty \int_{\R^d} |\phi(t)| N_\pi(t) dt. \]
\end{proposition}
\begin{proof}
Using the standard device of a partition of unity, it suffices to establish the inequality
\[ \left| \int_U \phi(\pi_1,\ldots,\pi_d) \psi(x) \left| \det \frac{\partial \pi}{\partial x} \right| dx \right| \leq \int_{\R^d} | \phi(t)|  \left( \sum_{x \in U \ : \ \pi(x) = t} |\psi(x)| \right) dt \]
where $U$ is a small open neighborhood in $\R^d$.  If the mapping $\pi$ were $1-1$ on $U$, then this equality would follow immediately from the change-of-variables formula.  Now the portion of $U$ for which $\det \frac{\partial \pi}{\partial x} \neq 0$ is an open set, hence by the implicit function theorem, this set may be decomposed into at most countably many pieces on which $\pi$ is $1-1$.  The portion of $U$ for which $\det \frac{\partial \pi}{\partial x} = 0$ contributes nothing to the integral since the integrand vanishes.  Thus, summing over the regions of $U$ on which $\pi$ is $1-1$ and using monotone convergence gives the desired inequality, since the sum of $|\psi(x)|$ over the points $x$ with $\pi(x) = t$ is trivially bounded by $||\psi||_\infty N_\pi(t)$.
\end{proof}

The key to understanding \eqref{multilin} in this context will be to take the functions $f_j$ to be characteristic functions and apply an iterative technique similar to the method of refinements introduced by Christ \cite{christ1998}.  In particular, whenever a Jacobian determinant appears, the sets under consideration will be decomposed into pieces on which the determinant is either above or below some threshold in magnitude.  The piece on which the determinant is large will be estimated using the proposition just established.  The piece on which it is small will be further decomposed based on the size of additional Jacobian determinants which arise.

To carry out this process, suppose $D$ is a closed subset of ${\cal M}$, and let $E := \set{ x \in D}{ \pi_i(x) \in E_i \ i=1,\ldots,m}$ for some measurable sets $E_1,\ldots,E_m \subset \R$. Fix $G := (i_1,\ldots,i_d)$ for some indices $i_1,\ldots,i_d \in \{1,\ldots,m\}$.  By definition of $\partial^G \pi$, $\left( \partial^G \pi \right) \omega = d \pi_{i_1} \wedge \cdots \wedge d \pi_{i_d}$.
Hence, by proposition \ref{cov},
\[  \int_{D} \chi_E \left| \partial^G \pi \right| \omega \leq \int_{\R^d} N_{(\pi_{i_1},\ldots,\pi_{i_d})}(t) \prod_{j=1}^d \chi_{E_{i_j}}(t_j) dt. \]
Assuming that $N_{(\pi_{i_1},\ldots,\pi_{i_d})}(t)$ is bounded by above by $N$, it follows that
\[   \int_{D} \chi_E \left| \partial^G \pi \right| \omega \leq N \prod_{j=1}^d |E_{i_j}|. \]
Now fix any $\delta > 0$ and let $E' := E \cap \set{x \in D}{ |\partial^{G} \pi(x)| \leq \frac{1}{2} \delta^{-1} \prod_{i=1}^m |E_i|^{G^{(i)}}}$.  On $E \setminus E'$, one has
\[ \int_{E \setminus E'} \omega \leq \int_{E \setminus E'} 2 \delta \prod_{j=1}^d |E_{i_j}|^{-1} |\partial^G \pi| \omega \leq 2 N \delta. \]
Thus it follows that
\[ \int_{\cal M} \chi_E \omega \leq 2 N \delta + \int_{\cal M} \chi_{E'} \omega. \]
Now $E'$ has the same form as $E$ with an additional constraint: the function $\partial^G \pi$ is constrained to take values in an interval centered at the origin of width $\delta^{-1} \prod_{j=1}^m |E_j|^{G^{(j)}}$.  Thus, by induction, one may restrict $E'$ further by applying the same procedure just followed again for any $d$-tuple of functions out of $\pi_1,\ldots,\pi_m$ and $\partial^G \pi$ for the particular $G \in {\cal G}_1$ previously chosen (provided, of course, that the multiplicity assumption is still satisfied).  Clearly any $d$-tree may be obtained in this manner after boundedly many steps (depending on the minimal $K$ for which it belongs to ${\cal G}_K$).  Thus, it follows that up to an error of size  $C_{K,m,d} N \delta$ in computing the $\omega$-measure of $E$, one may assume that $|\partial^G \pi| \leq \delta^{-\#G} \prod_{j=1}^m |E_j|^{G^{(j)}}$ for all $G \in {\cal G}_K^N$.  In other words, the following proposition has been established:
\begin{proposition}
There exists a constant $C$ depending only on $K, N, m,$ and $d$ such that, for any measurable sets $E_1,\ldots,E_m \subset \R$, any closed $D \subset {\cal M}$, and any $\delta > 0$, the set $E := \set{ x \in {D}}{ \pi_i(x) \in E_i, \ i =1,\ldots,m}$
 satisfies the inequality
\begin{equation}
\left| E \right|_\omega \leq C \delta + \left| E \cap \bigcap_{G \in {\cal G}_K^N ({\cal M},\pi)} \set{ x \in {\cal M}}{|\partial^G \pi(x)| \leq \delta^{-\#G} \prod_{j=1}^m |E_j|^{G^{(j)}}} \right|_\omega \label{mainprop}
\end{equation}
(where $|\cdot|_\omega$ indicates the $\omega$-measure of the corresponding set).
\end{proposition}
Theorem \ref{maintheorem} follows immediately from this proposition by setting $f_j = \chi_{E_j}$ for $j=1,\ldots,m$ and taking
\[ \delta := \left( \left(\inf_{x \in D} |\partial^G \pi(x)| \right)^{-1} \prod_{j=1}^m |E_j|^{G^{(j)}} \right)^{\frac{1}{\#G}}, \]
since, for this value of $\delta$, the set on the right-hand side of \eqref{mainprop} is empty.


\section{Proof of theorem \ref{oscthm}}
\label{oscsec}

The boundary of ${\cal M}$ played no major role in the proof of theorem \ref{maintheorem};  theorem \ref{oscthm}, on the other hand, is sensitive to the presence of a boundary; if ${\partial {\cal M}}$ is nonempty, it will be necessary to control the topology of various sublevel sets restricted to $\partial M$ as well.  To that end, some additional definitions concerning solutions of systems of equations on ${\cal M}$ are necessary:

\begin{definition}
Any $d$-tree $G$ will be said to satisfy the weak multiplicity $N$ condition on $ \partial {\cal M}$ when either $G$ is trivial or $G = (G_1,\ldots,G_d)$ and the system $\partial^{G_{\sigma_1}} \pi = c_1,\ldots, \partial^{G_{\sigma_{d-1}}} \pi = c_{d-1}$ has at most $N$ nondegenerate solutions on $\partial {\cal M}$ for any choice of $\sigma_1,\ldots,\sigma_{d-1}$ in $\{1,\ldots,d\}$. 
\end{definition}

\begin{definition}
Any $d$-tree $G$ will be said to satisfy the strong multiplicity $N$ condition on $\partial {\cal M}$ when either $G$ is trivial or $G$ satisfies the weak multiplicity $N$ condition on $ \partial {\cal M}$ and $G = (G_1,\ldots,G_d)$ for $G_1,\ldots,G_d$ also satisfying the strong multiplicity $N$ condition on $\partial {\cal M}$.  The collection of nontrivial $G \in {\cal G}_K$ which satisfy the strong multiplicity $N$ condition on ${\partial M}$ will be denoted ${\cal G}_K^N(\partial {\cal M},\pi)$.
\end{definition}
As in the previous section, analyticity and compactness imply that for any $K \neq 0$, $\bigcup_N {\cal G}_K^N(\partial {\cal M},\pi) = {\cal G}_K \setminus {\cal G}_0$, but the particular $N$ at which the collections stabilize can and generally does depend in some complex way on the functions $\pi_1,\ldots,\pi_m$.  In the case that ${\cal M}$ is a semi-Pfaffian set in $\R^d$, the finiteness theorem of Khovanski{\u\i} will again guarantee that this $N$ only depends on the formats of the functions in question as well as the complexity of the boundary of ${\cal M}$.  To be precise:
\begin{definition}
A semi-Pfaffian set is any set which may be written in the form
\[ D = \bigcup_{i=1}^N \set{ x \in U \subset \R^d}{f_{i1}(x) = \cdots = f_{iN}(x) = 0, g_{i1}(x) > 0, \ldots, g_{iN}(x) > 0} \]
where $f_{ij}$ and $g_{ij}$ are Pfaffian functions on $U$ with a common chain for $i,j = 1,\ldots,N$.  The format of $D$ is defined to be the collection of numbers including $N$, $d$, and the formats of $f_{ij}$ and $g_{ij}$ for each $i,j=1,\ldots,N$.
\end{definition}
Up to a set of dimension $d-1$, any such semi-Pfaffian (or, more generally, semi-analytic) set may be written as a union of sets of the form
\[ D_l := \set{ x \in {\cal M}}{\pi_{m+1,l}(x) \geq 0,\ldots,\pi_{M,l}(x) \geq 0} \]
for some functions $\pi_{j,l}$ where $j = m+1,\ldots,M$ and $l = 1,\ldots,M$, and these sets $D_l$ may be assumed to be non-overlapping (meaning that their intersection is at most $(d-1)$-dimensional). Now fix positive numbers $\epsilon_1,\ldots,\epsilon_{m-1}$, and let 
 \begin{align*}
D_\epsilon := & \set{ x \in {\cal M}}{ |\pi_j(x) | \leq \epsilon_j \ j=1,\ldots,m-1} \\
 & \hphantom{hhh} \cap  \set{x \in {\cal M}}{ \pi_j(x) \geq 0 \ j=m+1,\ldots,M}.
\end{align*}
If $D$ and $E_\epsilon$ are as defined in the statement of theorem \ref{oscthm}, it is clear that $D \cap E_\epsilon$ may be written as a non-overlapping union of sets of the form $D_\epsilon$ and that the number of such sets is controlled by the formats of $D$ and $\pi_1,\ldots,\pi_m$ in the semi-Pfaffian case.  To prove theorem \ref{oscthm}, then, it suffices to establish the following proposition:
\begin{proposition}
Given $(\pi_1,\ldots,\pi_M)$, fix a positive integer $K$ and let $N$ be such that ${\cal G}_{K+1}^N ({\cal M},\pi) \cap {\cal G}_{K+1}^N(\partial {\cal M},\pi) = {\cal G}_{K+1} \setminus {\cal G}_0$ for the $d$-trees on $M$ indices. Then for any nontrivial $G$ which is a $d$-tree on $m$ indices in ${\cal G}_K$ and all real $\lambda$,
\begin{equation}
 \left| \int_{D_\epsilon} e^{i \lambda \pi_m} \omega \right|  \leq C \left( \inf_{x \in D_\epsilon} |\partial^G \pi(x)| \right)^{-\frac{1}{\#G}} |\lambda|^{-\frac{G^{(m)}}{\#G}} \prod_{j =1}^{m-1} \epsilon_j^{\frac{G^{(j)}}{\# G}} \label{oscinq}
\end{equation}
for some constant $C$ that depends only on $N$, $K$, $M$, and $d$.
\end{proposition}
\begin{proof}
 To simplify matters somewhat, consider the integral
\begin{equation}  \int \left( e^{i \lambda \pi_m} \prod_{j=1}^{M} \eta_j \circ \pi_j  \right) \omega \label{oscint1}
\end{equation}
where $|\eta_j(t)| \leq 1$ and $\int |\eta'_j(t)| dt \leq 2$ for all $j$ and  $\int |\eta_j(t)|  dt \leq 2 \epsilon_j$
when $j < m$.  This proposition will establish uniform bounds for \eqref{oscint1} over the entire class of $\eta_j$'s satisfying the stated inequalities.  By a standard limiting argument, the inequality \eqref{oscinq} will follow.

Let $s$ be an element of $\R^S$ for some (presumably large) $S$, and let the coordinates of $s$ be indexed by the variable $\sigma$ ($\sigma$ will belong to some index set of cardinality $S$ to be specified momentarily).  Suppose $\delta \in \R^S$ has nonnegative entries.  For each $\sigma$, consider the function given by
\[ \psi_\sigma(s) := \frac{\delta_\sigma^2 s_\sigma}{\sum_{\sigma'} \delta_{\sigma'}^2 s_{\sigma'}^2} \left( 1 - \prod_{\sigma'} \eta_0(\delta_{\sigma'} s_{\sigma'})  \right),\]
where $\eta_0$ is supported on $[-1,1]$, identically one on $[-\frac{1}{2},\frac{1}{2}]$, bounded in magnitude by $1$ and has a continuous derivative which satisfies $\int |\eta_0'(t)| dt \leq 2$.
Since $\psi_{\sigma}(s) = 0$ whenever $|\delta_\sigma s_\sigma| \leq \frac{1}{2}$ for each $\sigma$, it follows that there is a constant $C$ independent of $\delta$ and $s$ such that, for any indices $\sigma,\sigma'$
\begin{align} 
|\psi_\sigma(s)|  & \leq C \delta_\sigma, \label{decay1} \\
 |\partial_{\sigma'} \psi_{\sigma}(s)| & \leq \frac{C \delta_{\sigma} \delta_{\sigma'}}{1 + \delta_{\sigma'}^2 s_{\sigma'}^2}. \label{decay2}
\end{align}
Given these estimates, consider now the special case when $\sigma$ is is contained in the index set given by all possible subsets of $\{1,\ldots,m-1\}$ which have cardinality $d-1$ (that is, $\sigma$ will be one particular subset of $\{1,\ldots,m-1\}$).  For any such $\sigma$, let $\sigma_1, \sigma_2,\ldots,\sigma_{d-1}$ be the elements of $\sigma$ arranged in increasing order.  Now consider the $(d-1)$-form on ${\cal M}$ given by
\begin{equation}
 e^{i \lambda \pi_m} \prod_{j=1}^M \eta_j \circ \pi_j \sum_{\sigma} \psi_{\sigma} ( \partial^{(m,\cdot)} \pi ) d\pi_{\sigma_1} \wedge \cdots \wedge d \pi_{\sigma_{d-1}}, \label{boundaryform}
\end{equation}
where $\partial^{(m,\cdot)} \pi$ represents the vector whose $\sigma'$-th coordinate is $\partial^{(m,\sigma')} \pi$, i.e., $\partial^{(m,\sigma'_1,\ldots,\sigma'_{d-1})} \pi$.
By proposition \ref{covprop}, if the number of nondegenerate solutions on $\partial {\cal M}$ of the system \[ (\pi_{\sigma_1}(x),\ldots,\pi_{\sigma_{d-1}}(x)) = (c_1,\ldots,c_{d-1})\] is bounded by $N$, then the integral of \eqref{boundaryform} on $\partial {\cal M}$ will be bounded by a dimensional constant times $N \sum_{\sigma} \epsilon_{\sigma_1} \cdots \epsilon_{\sigma_{d-1}} \delta_{\sigma}$ (the $\epsilon$'s come from the integrals of $|\eta_j|$ for $j < m$ and the $\delta_\sigma$ comes from \eqref{decay1}).

This exterior derivative of \eqref{boundaryform}, on the other hand, is easily calculated; the result is a $d$-form equal to $e^{i \lambda \pi_m } (\prod_{j=1}^M  \eta_j \circ \pi_j ) (\Psi_1 + \Psi_2 + \Psi_3) \omega$, where $\omega$ is the $d$-form attached to ${\cal M}$ and $\Psi_1,\ldots,\Psi_3$ are the functions given by
\begin{align*}
\Psi_1(x) & := i \lambda \sum_\sigma \partial^{(m,\sigma)} \pi (x) \psi_{\sigma} ( \partial^{(m,\cdot)} \pi(x)) \\
\Psi_2(x) & := \sum_{j=1}^M \sum_{\sigma} \frac{\eta_j' \circ \pi_j(x)}{\eta_j \circ \pi_j(x)} \partial^{(j,\sigma)} \pi(x) \psi_{\sigma} ( \partial^{(m,\cdot)} \pi(x)) \\
\Psi_3(x) & := \sum_{\sigma,\sigma'} \partial^{((m,\sigma'),\sigma)} \pi(x) (\partial_{\sigma'} \psi_\sigma)(\partial^{(m,\cdot)} \pi(x))
\end{align*}
The first function, $\Psi_1$, simply equals $i \lambda (1 - \prod_\sigma \eta_0(\delta_\sigma \partial^{(m,\sigma)} \pi(x) ))$.  Each term in the sum defining $\Psi_2$ may be estimated by means of proposition \ref{covprop} since the derivative $\partial^{(j,\sigma)} \pi$ corresponds to the Jacobian determinant of the transformation $x \mapsto (\pi_j(x),\pi_{\sigma_1}(x),\ldots,\pi_{\sigma_{d-1}}(x))$.  Estimating the $L^\infty$ norm of $\psi_\sigma$ with the upper bound \eqref{decay1} gives
\begin{align*} 
\left| \int e^{i \lambda \pi_m} \left(\prod_{k=1}^M \eta_k \circ \pi_k \right) \right. & \left. \partial^{(j,\sigma)} \pi \frac{\eta_j' \circ \pi_j}{\eta_j \circ \pi_j}   \psi_\sigma(\partial^{(m,\cdot)} \pi ) \omega \right|   \\
& \leq C N \delta_\sigma \int |\eta_j'(t)| dt \prod_{k=1}^{d-1} \int |\eta_{\sigma_k}(t)| dt \\
& \leq C' N \epsilon_{\sigma_1} \cdots \epsilon_{\sigma_{d-1}} \delta_\sigma
\end{align*}
where $C'$ is some dimensional constant and $N$ is a bound on the number of nondegenerate solutions of the system $(\pi_j(x),\pi_{\sigma_1}(x),\ldots,\pi_{\sigma_{d-1}}(x)) = (c_1,\ldots,c_d)$. The same procedure holds for all terms in the third sum (defining $\Psi_3$); this time, the change-of-variables involves the function $\partial^{(m,\sigma')} \pi $ and $\pi_{\sigma_1},\ldots,\pi_{\sigma_{d-1}}$.  In this case, the estimate \eqref{decay2} gives that 
\begin{align*}
\left| \int e^{i \lambda \pi_m} \left( \prod_{k=1}^M \eta_k \circ \pi_k \right)  \right. & \left. \partial^{((m,\sigma'),\sigma)} \pi  (\partial_{\sigma'} \psi_{\sigma})(\partial^{(m,\cdot)} \pi) \omega \right|  \\
& \leq C N \int \frac{\delta_{\sigma} \delta_{\sigma'}}{1 + \delta_{\sigma'} t^2} dt \prod_{k=1}^{d-1} \int |\eta_{\sigma_k}(t)| dt \\
& \leq C' N \epsilon_{\sigma_1} \cdots \epsilon_{\sigma_{d-1}} \delta_\sigma.
\end{align*}
Now fix a single constant $\delta$, and for each $\sigma$, let $\delta_\sigma := \lambda \delta (\epsilon_{\sigma_1} \cdots \epsilon_{\sigma_{d-1}})^{-1}$.  Finally, Stokes' theorem may be applied to relate the integral of \eqref{boundaryform} on $\partial {\cal M}$ to the integral of its exterior derivative on ${\cal M}$.  This equality, combined with the equality for $\Psi_1$ and the estimates for $\Psi_2$ and $\Psi_3$, gives that
\begin{align*}
 & \left| \int_{\cal M} e^{i \lambda \pi_m} \left( \prod_{j=1}^M \eta_j \circ \pi_j \right) \prod_{\sigma} \eta_0(\lambda \delta \epsilon_{\sigma_1}^{-1} \cdots \epsilon_{\sigma_d}^{-1} \partial^{(m,\sigma)}\pi) \omega  \right. \\
& \qquad - \left. \int_{\cal M} e^{i \lambda \pi_m} \left( \prod_{j=1}^M \eta_j \circ \pi_j \right) \omega \right| \leq {C N \delta}. 
\end{align*}
Thus, at the cost of an error term of size $C N \delta$, one may add a smooth cutoff to \eqref{oscint1} restricting the integral to the region where $|\partial^{(m,\sigma)} \pi| \leq \delta^{-1} \lambda^{-1} \prod_{j=1}^{m-1} \epsilon_{\sigma_j}$ is small whenever $\sigma$ is a cardinality $d-1$ subset of $\{1,\ldots,m-1\}$.  Moreover, at the same cost, the domain of integration may be restricted to the set where the Jacobian determinant of any $d$-tuple of these functions is small (for the same reason as in the proof of theorem \ref{maintheorem}, namely, the support of the integral must be small when the Jacobian is large).  Just as in \ref{maintheorem}, a trivial induction on ${\cal M}$ finishes the proposition.
\end{proof}
\section{Final remarks}








\begin{enumerate}
\item Unlike Carbery, Christ, and Wright, the function $F$ in theorem \ref{ccwthm1} is real analytic or Pfaffian, but need not satisfy $L F(x) \geq 1$ on the whole unit box for the theorem to be valid.  It is easy to show that hypotheses closer to the original (namely, smooth $F$ which satisfy a derivative condition everywhere on the unit box) can fail to imply results in this more general situation.  The most straightforward example comes from the family of functions $F_N(x,y) = N^{-1} e^{x} \sin Ny$ on $[0,1]^2$.  As $N \rightarrow \infty$, the determinant of the Hessian of $F_N$ is uniformly bounded away from $0$ on $[0,1]^2$, but clearly no nontrivial uniform sublevel set estimates hold for this family. 
\item The gain in decay in theorem \ref{ccwthm1} is consistent with the limitations observed by Phong, Stein, and Sturm since $L (p \circ \ell) \equiv 0$ for any $L \in {\cal L}^{a,b}$ with $a \geq 2$ when $p$ is a function of one variable and $\ell$ is linear.  The reason for this is that the gradients of any $\partial^\beta (p \circ \ell)$ all point in the same direction.  Moreover, if $F$ is not of this form, then it will not be the case that $L F \equiv 0$ for all $L \in {\cal L}^{a,b}$ with $a \geq 2$.  This is reminiscent of work by Robert and Sargos \cite{rs2005}, in which they showed that, when $F$ is a polynomial of degree $k$ in $x_1,\ldots,x_d$, the decay rate of the oscillatory integral
\[ \int \varphi(x) e^{i \lambda F(x)} dx \]
(where $\varphi$ continuously differentiable and compactly supported) is at least $\lambda^{-1/(k-1)}$ unless $F$ happens to be of the form $F(x) = c_0 + c_1 (\ell(x))^k$ for some linear function $\ell$.
\item Theorem \ref{pssthm2} is sharp in the sense that, for a given $L \in {\cal L}^{a,b}$, scaling dictates that the supremum of
\[ \epsilon^{-\sigma} \prod_{i=1}^d |E_i|^{-1 + \tau_i} \int_D \chi_{|F(x)| \leq \epsilon} \prod_{i=1}^d \chi_{E_i}(x_i) dx \]
over all $E_1,\ldots,E_d$ and all $D, F$ with $L F \geq 1$ on $D$, and all $\epsilon > 0$ must be infinite unless $(\sigma,\tau_1,\ldots,\tau_d) = \theta (a,b_1,\ldots,b_d)$ for some constant $\theta$.  Aside from the case $a=1$, though, the complexity of the ${\cal L}^{a,b}$ makes it difficult to determine whether $\theta = \frac{1}{|b|+1-a}$ is the largest possible in general.  The same remarks are valid for theorem \ref{maintheorem} as well.
\end{enumerate}

\section{Acknowledgments}

The author wishes to thank A. Gabrielov for information and references relating to the topology of semi-analytic sets.  This research was partially supported by  NSF grant DMS-0850791.




\bibliography{mybib}

\end{document}

%% file: header.tex
\usepackage{amsmath}
\usepackage{amsthm}
\usepackage{amsfonts}
\usepackage{amsbsy}
\usepackage{amssymb}
\newtheorem{theorem}{Theorem}
\newtheorem{proposition}{Proposition}

\newtheorem{definition}{Definition}

\newcommand{\R}{{\mathbb R}}

\newcommand{\set}[2]{ \left\{ #1 \ \left| \ #2 \right. \right\} }